\documentclass[a4paper,10pt, reqno]{amsart}

\usepackage[english]{babel}
\usepackage[utf8]{inputenc}

\usepackage{amssymb}
\usepackage{amsmath}
\usepackage{amsthm}
\usepackage{bbm}

\usepackage{tikz}

\newcommand{\bbC}{\mathbb{C}}

\newcommand{\bbR}{\mathbb{R}}

\newcommand{\calL}{\mathcal{L}}
\newcommand{\calR}{\mathcal{R}}

\newcommand{\re}{\operatorname{Re}}

\theoremstyle{definition}
\newtheorem{definition}{Definition}[section]
\newtheorem{notation}[definition]{Notation}
\newtheorem{remark}[definition]{Remark}
\newtheorem{remarks}[definition]{Remarks}

\theoremstyle{plain}
\newtheorem{proposition}[definition]{Proposition}
\newtheorem{lemma}[definition]{Lemma}
\newtheorem{theorem}[definition]{Theorem}

\numberwithin{equation}{section}
\numberwithin{figure}{section}

\begin{document}

\title{A note on approximation of operator semigroups}
\author{Jochen Gl\"uck}
\email{jochen.glueck@uni-ulm.de}
\address{Jochen Gl\"uck, Institute of Applied Analysis, Ulm University, 89069 Ulm, Germany}
\date{January 26, 2015}
\begin{abstract}
	Let $A$ be a bounded linear operator and $P$ a bounded linear projection on a Banach space $X$. We show that the operator semigroup $(e^{t(A-kP)})_{t \ge 0}$ converges to a semigroup on a subspace of $X$ as $k \to \infty$ and we compute the limit semigroup.
\end{abstract}
\maketitle

\section{Introduction and discussion of the main result}

To motivate the content of this note, let $(\Omega, \Sigma, \mu)$ be a $\sigma$-finite measure space and $(e^{tA})_{t \ge 0}$ a positive $C_0$-semigroup on $X := L^p(\Omega,\Sigma,\mu)$ for some $p \in [1,\infty)$. If $B \subset \Omega$ is a measurable set, if $P$ is the projection on $X$ which is given by multiplication with the indicator function $\mathbbm{1}_B$ and if $Q := 1-P$, then for all $t > 0$ the limits
\begin{align}
	\lim_{k\to \infty} e^{t(A-kP)} \qquad \text{and} \qquad \lim_{k \to \infty} \big( e^{\frac{t}{k}A}Q \big)^k \label{form:absorption-for-positive-semigroups}
\end{align}
exist with respect to the strong operator topology and coincide \cite[Lemma~4.1 and Theorem~5.3]{Arendt1993}. The limit can be shown to be a certain degenerate operator semigroup \cite[p.~431--432]{Arendt1993} which we might refer to as a sort of \emph{absorption semigroup}. Usually, the latter term is used to describe the semigroup governed by an abstract Cauchy problem which models a diffusion process with a potential (= absorption term). For the connection of such absorption semigroups to the limit semigroup in~(\ref{form:absorption-for-positive-semigroups}) we refer the reader to \cite{Arendt1993}; see also \cite{Voigt1986, Voigt1988} for some additional background information. 

When considering the above convergence result, a number of potential generalisations immediately comes into ones mind; for example, one could try to replace the projection $P$, which is given by multiplication with an indicator function, by a more general projection. One could also try to consider more general spaces and/or to omit the positivity assumption. We refer to \cite{Arendt1997, Matolcsi2003, Matolcsi2003a, Matolcsi2004} and the references therein for a number of positive and negative results in this direction. It is also worthwhile pointing out that the second limit in~(\ref{form:absorption-for-positive-semigroups}) is closely related to the so-called \emph{Quantum Zeno Effect}; we refer to \cite{Exner2007, Facchi2010} and the references therein for more details.

In this note we are concerned with the existence of the limits in~(\ref{form:absorption-for-positive-semigroups}) in the case where $A$ is a \emph{bounded} linear operator (without any additional properties) on an arbitrary Banach space and where $P$ is an arbitrary bounded linear projection. For this case, the second limit in~(\ref{form:absorption-for-positive-semigroups}) was shown by Matolcsi and Shvidkoy to always exist with respect to the strong operator topology; moreover they identified the limit semigroup and thus proved the following theorem \cite[Theorem~1]{Matolcsi2003}:

\begin{theorem}[Matolcsi, Shvidkoy] \label{thm:matolcsi-shvidkoy}
	Let $A$ be a bounded linear operator and $Q$ a bounded linear projection on a Banach space $X$. For each $x \in X$ and each $t \ge 0$ we have 
	\begin{align*}
		\big( e^{\frac{t}{k}} Q \big)^kx \to e^{tQAQ}Qx \quad \text{as } k \to \infty,
	\end{align*}
	where the convergence is uniform with respect to $t$ on bounded subsets of $[0,\infty)$.
\end{theorem}

As the generator $A$ is bounded it is natural to ask whether the limit exists even in the operator norm topology. In Section~\ref{section:lifting-to-operator-norm-convergence} at the end of this note we will prove that this is indeed true. Our focus, however, is on another question, namely whether the limit $\lim_{k \to \infty} e^{t(A-kP)}$ always exists if $A$ is bounded. Our main result gives an affirmative answer to this question:

\begin{theorem} \label{thm:main-theorem}
	Let $A$ be a bounded linear operator and $P$ a bounded linear projection on a complex Banach space $X$. Define $Q := 1 - P$ and let $z \in \bbC$. For all $t > 0$ we have
	\begin{align*}
		e^{t(A+zP)} \to e^{tQAQ}Q \quad \text{as} \quad \re z \to -\infty
	\end{align*}
	with respect to the operator norm, and the convergence is uniform with respect to $t$ on compact subsets of $(0,\infty)$.
\end{theorem}

The proof of Theorem~\ref{thm:main-theorem}, which we present in Section~\ref{section:proof-of-main-result}, uses only elementary methods, but it requires a careful analysis of the spectral properties of $A+zP$.

In fact we will obtain a bit more information about the convergence in the above theorem, including an explicit estimate (for undefined notation we refer to the end of the introduction):

\begin{remark} \label{rem:uniform-convergence-and-decay-rate}
	For $t$ in any fixed compact subset of $(0,\infty)$, the convergence in Theorem~\ref{thm:main-theorem} has at least a linear rate. 
	
	More precisely, the following holds: Let $0 < T_1 < T_2$ and define $R := 2(\|A\| + \delta) \|P-Q\|$, where $\delta > 0$ is any number sufficiently large to ensure that $R$ is strictly larger than the spectral radius of $QAQ$. Then for every $z \in \bbC$ with $\re z < -2R$ and all $t \in [T_1,T_2]$ we have
	\begin{align*}
		\|e^{t(A+zP)} - e^{tQAQ}Q\| \le C_1 e^{T_1\re z} + C_2 \frac{1}{|z| - R},
	\end{align*}
	where $C_2 := Re^{T_2R} \, \frac{\|A\| + \delta}{\delta} \|P\| \, \sup_{|\lambda| = R} \|I + AQ \, \calR(\lambda,QAQ)\|$ and $C_1 := \frac{R e^{T_1R}}{\delta}$; here, $\calR(\lambda,QAQ)$ denotes the resolvent of $QAQ$ at $\lambda$. Note that $C_2$ can be further estimated by means of the Neumann series if we chose $\delta > 0$ sufficiently large such that even $R > \|QAQ\|$.
\end{remark}

A few further remarks are in order.

\begin{remarks}
	(a) Using complexifications one immediately obtains an analogues result for real Banach spaces.
	
	(b) The operator family $(e^{tQAQ}Q)_{t \ge 0}$ is a (norm continuous) $C_0$-semi\-group on $QX = \ker P$. Hence, the semigroup $(e^{t(A+zP)})_{t \ge 0}$ converges to $0$ on the range of $P$ and to a semigroup with generator $QAQ$ on the kernel of $P$.
	
	(c) Using the power series expansion of the exponential function one obtains several ways to write down the limit semigroup; in fact, one has
	\begin{align*}
		e^{tQAQ}Q = e^{tQA}Q = Qe^{tAQ} = Qe^{tQAQ} = Qe^{tQAQ}Q
	\end{align*}
	for every $t \in (0,\infty)$.
	
	(d) The assertion of Theorem~\ref{thm:main-theorem} is obviously false for $t \le 0$; just consider $X = \bbC$, $A = 0$ and $P = 1$ to see this.
	
	(e) A glance at the proof of Theorem~\ref{thm:main-theorem} in the next section reveals that one can show similar results for analytic functions $f$ other than $\exp$, provided that $f$ satisfies certain decay properties. Since our focus is on operator semigroups, we shall not discuss this in detail.
\end{remarks}

It seems to be unclear what happens to the assertion of Theorem~\ref{thm:main-theorem} if we consider $C_0$-semigroups with \emph{unbounded} generator $A$. Of course one expects that some additional conditions are necessary in this case, since otherwise it might happen that $QAQ$ has only a very small domain which need not be dense in the range of $Q$; moreover, $QAQ$ might not even be closed. Concerning the related Theorem~\ref{thm:matolcsi-shvidkoy}, we also point out that the strong limit $\lim_{k \to \infty}\big( e^{\frac{t}{k}A}Q \big)^k$ \emph{never} exists for all bounded linear projections $Q$ on $X$ unless $A$ is bounded \cite[Theorem~2.1]{Matolcsi2004}. In any case, the proof of Theorem~\ref{thm:main-theorem} which we present below relies heavily on the boundedness of $A$. We therefore leave it as an open problem to analyse the case of unbounded generators.

Before we give a proof of our main result in the next section, let us briefly fix some notation. If $X$ is a complex Banach space, then we denote by $\calL(X)$ the space of all bounded linear operators on $X$. The spectrum of an operator $A \in \calL(X)$ is denoted by $\sigma(A)$ and for every $\lambda \in \bbC \setminus \sigma(A)$, $\calR(\lambda,A) := (\lambda-A)^{-1}$ denotes the resolvent of $A$ at $\lambda$. For $r \ge 0$ and $z \in \bbC$ we denote by $B_r(z) := \{\mu \in \bbC| \, |\mu-z| < r\}$ the open disk in $\bbC$ with center $z$ and radius $r$ and by $\overline{B}_r(z) := \{\mu \in \bbC| \, |\mu-z| \le r\}$ the closed disk in $\bbC$ with center $z$ and radius $r$.

\section{Proof of the main result} \label{section:proof-of-main-result}

In this section we prove our main result. Throughout the section we may assume that $X \not= \{0\}$ and we let $0 < T_1 < T_2$. We want to find a norm estimate for the difference $e^{t(A+zP)} - e^{tQAQ}Q$, where $t \in [T_1,T_2]$ and where $\re z \to - \infty$. To this end we employ the functional calculus for analytic functions and write
\begin{align}
		e^{t(A+zP)} - e^{tQAQ}Q = \frac{1}{2\pi i} \oint_\gamma e^{t\lambda}[\calR(\lambda,A+zP) - \calR(\lambda,QAQ)Q] \, d\lambda  \label{form:fc-representation}
\end{align}
for an appropriate path $\gamma$ which encircles both the spectra of $A+zP$ and $QAQ$. Of course we could choose $\gamma$ to be a sufficiently large circle, but this is too crude to obtain a reasonable estimate. Hence, our first task is to localize the spectrum of $A+zP$ more precisely in order to find a good choice for $\gamma$.

To do this, we first show a simple auxiliary result about the resolvent of our projection $P$.

\begin{lemma} \label{lem:resolvent-of-projection}
	Let $z \in \bbC$.
	\begin{itemize}
		\item[(a)] We have $\sigma(zP) \subset \{0,z\}$ and for every $\lambda \in \bbC \setminus \{0,z\}$, the resolvent of $zP$ is given by
			\begin{align}
				\calR(\lambda,zP) & = \frac{\lambda - zQ}{\lambda(\lambda-z)}  \label{form:resolvent-of-projection-short} \\
				& = \frac{1}{2} \big( \frac{1}{\lambda} + \frac{1}{\lambda - z} + \frac{z(P-Q)}{\lambda(\lambda-z)} \big). \label{form:resolvent-of-projection-long}
			\end{align}
		\item[(b)] Consider a non-negative parameter $\alpha \ge 0$ and the radius $r_\alpha := 2\alpha \|P-Q\|$. If $\lambda \in \bbC$ is contained in none of the closed disks $\overline{B}_{r_\alpha}(0)$ and $\overline{B}_{r_\alpha}(z)$, then the resolvent $\calR(\lambda,zP)$ fulfils the estimate
			\begin{align*}
				\alpha \|\calR(\lambda,zP)\| < 1.
			\end{align*}
	\end{itemize}
\end{lemma}

Loosely speaking, the estimate in (b) says that the resolvent $\calR(\lambda,zP)$ decreases linearly as $\lambda$ tends away from the points $0$ and $z$.

\begin{proof}[Proof of Lemma~\ref{lem:resolvent-of-projection}]
	(a) We clearly have $\sigma(zP) \subset \{0,z\}$ since $\sigma(P) \subset \{0,1\}$. The representation formulas for the resolvent can be verified by a simple computation.
	
	(b) Assume that $\lambda \not\in \overline{B}_{r_\alpha}(0) \cup \overline{B}_{r_\alpha}(z)$; then $|\lambda| > r_\alpha \ge 2\alpha$ and $|\lambda - z| > r_\alpha \ge 2\alpha$. Moreover, at least one of the numbers $|\lambda|$ and $|\lambda-z|$ is no less than $\frac{|z|}{2}$ since we have $|z| \le |\lambda| + |\lambda - z|$. Since the other one is no less than $r_\alpha = 2\alpha \|P-Q\|$, we conclude that $|\lambda| |\lambda - z| \ge \alpha |z| \|P-Q\|$. Hence we conclude from the resolvent representation formula~(\ref{form:resolvent-of-projection-long}) in (a) that
	\begin{align*}
		\alpha \|\calR(\lambda,zP)\| \le \frac{1}{2} \big( \frac{\alpha}{|\lambda|} + \frac{\alpha}{|\lambda-z|} + \frac{\alpha |z| \|P-Q\|}{|\lambda| |\lambda - z|} \big) < 1,
	\end{align*}
	which proves the assertion.
\end{proof}

Now we can achieve our first goal and obtain very precise information on the position of the spectrum $\sigma(A+zP)$; we also obtain a Neumann type series representation for the resolvent:

\begin{proposition} \label{prop:spectrum-of-shifted-generator}
	Let $z \in \bbC$ and consider the radius $r := 2 \|A\| \|P-Q\|$. 
	\begin{itemize}
		\item[(a)] The spectrum of $A + zP$ is contained in the union of the two closed disks $\overline{B}_r(0)$ and $\overline{B}_r(z)$.
		\item[(b)] If $\lambda \in \bbC$ is contained in none of the discs $\overline{B}_r(0)$ and $\overline{B}_r(z)$, then
			\begin{align*}
				\calR(\lambda,A+zP) = \sum_{k=0}^\infty \big(\calR(\lambda,zP)A\big)^k \calR(\lambda,zP),
			\end{align*}
			where the series converges absolutely with respect to the operator norm.
	\end{itemize}
	\begin{proof}
		For every $\lambda \not\in \overline{B}_r(0) \cup \overline{B}_r(z)$ the equation
		\begin{align*}
			& \lambda - (A+zP) = (\lambda - zP) - A = (\lambda - zP) \big( 1 - \calR(\lambda,zP) A \big).
		\end{align*}
		holds. If we apply Lemma~\ref{lem:resolvent-of-projection}(b) with the parameter $\alpha := \|A\|$, then we obtain the estimate $\|A\| \|\calR(\lambda,zP)\| < 1$, and thus (a) and (b) follow by employing the Neumann series.
	\end{proof}
\end{proposition}

Part (a) of the above proposition suggests a strategy to estimate the expression in~(\ref{form:fc-representation}): we rewrite the contour integral in~(\ref{form:fc-representation}) as a sum of contour integrals around $\overline{B}_r(0)$ and $\overline{B}_r(z)$; the path of integration around the first disk should also be sufficiently large to encircle the spectrum of $QAQ$. Let us therefore introduce the following notation:

\begin{notation}
	For the rest of this section we use the following notation:
	\begin{itemize}
		\item[(a)] Let $r := 2 \|A\| \|P-Q\|$ as in Proposition~\ref{prop:spectrum-of-shifted-generator}. 
		\item[(b)] Let $R := 2(\|A\| + \delta) \|P-Q\| > r$, where $\delta > 0$ is chosen sufficiently large to ensure that the open disk $B_R(0)$ contains the spectrum of $QAQ$. 
	\end{itemize}
\end{notation}

If $|z| > 2R$, then the circles with radius $R$ around $0$ and $z$ do not intersect; using the information about the spectrum of $A+zP$ that we obtained in Proposition~\ref{prop:spectrum-of-shifted-generator}(a), we can therefore rewrite~(\ref{form:fc-representation}) as
\begin{align}
	\begin{aligned}
	e^{t(A+zP)} - e^{tQAQ}Q \quad = \quad & \frac{1}{2\pi i} \oint_{|\lambda-z| = R} e^{t\lambda}\calR(\lambda,A+zP) \, d\lambda \\
	+ & \frac{1}{2\pi i} \oint_{|\lambda| = R} e^{t\lambda}[\calR(\lambda,A+zP) - \calR(\lambda,QAQ)Q] \, d\lambda, 
	\end{aligned}
	\label{form:fc-representation-better}
\end{align}
where the paths of integration are parametrized with positive orientation. The spectra of $A+zP$ and $QAQ$ and the paths of integration in formula~(\ref{form:fc-representation-better}) are shown in Figure~\ref{fig:countors-of-integrals}.

\begin{figure}[ht]
	\centering
	\begin{tikzpicture}[scale=.6,font=\footnotesize]
		\newcommand{\rsmall}{1.5}
		\newcommand{\rlarge}{2.5}
		\newcommand{\zright}{-10}
		\newcommand{\zup}{1}
		%
		\filldraw[lightgray]  plot[smooth, tension=.7] coordinates {(0.6*\rsmall,0*\rsmall) (0.5*\rsmall,0.6*\rsmall) (0*\rsmall,0.45*\rsmall) (-0.4*\rsmall,0.6*\rsmall) (-0.5*\rsmall,-0.4*\rsmall) (-0.1*\rsmall,-0.3*\rsmall) (0.5*\rsmall,-0.4*\rsmall) (0.6*\rsmall,0*\rsmall)};
		\draw[gray] (-0.28*\rsmall,0.4*\rsmall) -- (\zright*0.5+0.25,1.4*\rlarge);
		\draw (\zright*0.5,1.4*\rlarge) node[above] {$\sigma(A+zP)$};
		%
		\filldraw[lightgray]  plot[smooth, tension=.7] coordinates {(\zright+0.5*\rsmall,\zup+0*\rsmall) (\zright+0.5*\rsmall,\zup+0.5*\rsmall) (\zright-0.4*\rsmall,\zup+0.6*\rsmall) (\zright-0.7*\rsmall,\zup-0.3*\rsmall) (\zright-0.1*\rsmall,\zup-0.3*\rsmall) (\zright+0.5*\rsmall,\zup-0.5*\rsmall) (\zright+0.5*\rsmall,\zup+0*\rsmall)};
		\draw[gray] (\zright+0.333*\rsmall,\zup+0.333*\rsmall) -- (\zright*0.5-0.25,1.4*\rlarge);
		%
		\draw[gray] plot[smooth, tension=.7] coordinates {(-1.15*\rsmall,0.35*\rsmall) (-0.9*\rsmall,-0.15*\rsmall) (-0.2*\rsmall,-0.7*\rsmall) (0.3*\rsmall,-1.25*\rsmall) (1.1*\rsmall,-0.2*\rsmall) (0.25*\rsmall,0.15*\rsmall) (-1.15*\rsmall,0.35*\rsmall)};
		\draw[gray] (0.3*\rsmall,-0.7*\rsmall) -- (-1.4*\rlarge,-0.85*\rlarge);
		\draw (-1.4*\rlarge,-0.85*\rlarge) node[left] {$\sigma(QAQ)$};
		\draw[->] (-14,0) -- (4,0) node[below] {$\bbR$};
		\draw[->] (0,-3) -- (0,5) node[left] {$i\bbR$};
		\draw (0,0) circle (\rsmall);
		\draw[dashed] (0,0) circle (\rlarge);
		\draw (0,0) -- (0.5*\rsmall,0.866*\rsmall) node[above] {$r$};
		\draw (0,0) -- (0.866*\rlarge,0.5*\rlarge) node[right] {$R$};
		\filldraw (\zright,\zup) circle (0.05);
		\draw (\zright,\zup) node[below] {$z$};
		\draw (\zright,\zup) circle (\rsmall);
		\draw[dashed] (\zright,\zup) circle (\rlarge);
		\draw (\zright,\zup) -- (\zright -0.5*\rsmall, \zup + 0.866*\rsmall) node[above] {$r$};
		\draw (\zright,\zup) -- (\zright -0.866*\rlarge,\zup + 0.5*\rlarge) node[left] {$R$};
	\end{tikzpicture}
	\caption{Spectra of $A+zP$ and $QAQ$; the dashed circles depict the paths of integration in formula~(\ref{form:fc-representation-better}).}
	\label{fig:countors-of-integrals}
\end{figure}
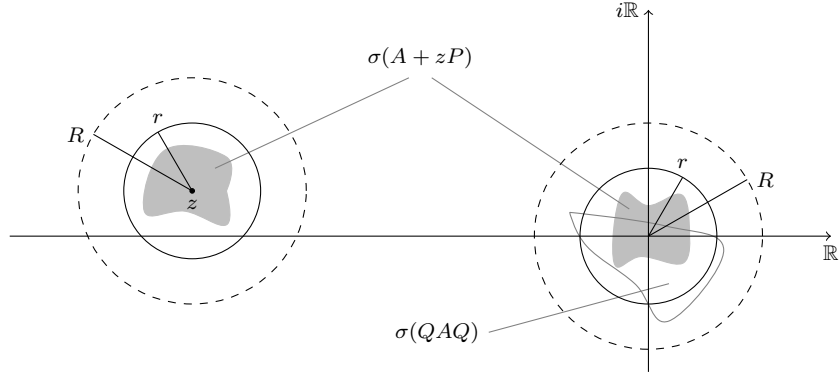

Our goal is now to estimate both integrals in formula~(\ref{form:fc-representation-better}). Let us start with the first integral, which turns out to be rather easy: since $t \ge T_1$, the exponential term within the integral ensures a fast decay as $\re z \to - \infty$ provided that we can control the resolvent. This is the content of the following proposition.

\begin{proposition} \label{prop:estimate-on-outer-cirlce}
	Let $z \in \bbC$, $\re z < -2R$.
	\begin{itemize}
		\item[(a)] We have $\|\calR(\lambda, A+zP)\| \le \frac{1}{\delta}$ for all $\lambda \in \bbC$ with $|\lambda-z| = R$.
		\item[(b)] For all $t \in [T_1,\infty)$ we have
			\begin{align*}
				\|\frac{1}{2\pi i} \oint_{|\lambda-z| = R} e^{t\lambda}\calR(\lambda,A+zP) \, d\lambda\| \le \frac{R e^{T_1(\re z + R)}}{\delta}.
			\end{align*}
	\end{itemize}
	\begin{proof}
		(a) To prove (a) we do not really need that $\re z < -2R$ but only that the disks $\overline{B}_R(0)$ and $\overline{B}_R(z)$ do not intersect. Let $\lambda \in \bbC$ with $|\lambda - z| = R$. In Proposition~\ref{prop:spectrum-of-shifted-generator}(b) we proved the resolvent representation formula 
		\begin{align*}
			\calR(\lambda,A+zP) = \sum_{k=0}^\infty \big(\calR(\lambda,zP)A\big)^k \calR(\lambda,zP).
		\end{align*}
		From Lemma~\ref{lem:resolvent-of-projection}(b) we obtain (with $\alpha := \|A\| + \delta$ and by approximating the circle $\partial B_{r_\alpha}(z)$ from the outside) the estimate $(\|A\| + \delta) \|\calR(\lambda,zP)\| \le 1$. Plugging this into the above representation formula for $\calR(\lambda,A+zP)$ we compute
		\begin{align*}
			\|\calR(\lambda,A+zP)\| \le \sum_{k=0}^\infty \big( \frac{\|A\|}{\|A\| + \delta} \big)^k \frac{1}{\|A\| + \delta} = \frac{1}{\delta}.
		\end{align*}
		
		(b) Assertion (b) readily follows from (a) since $\re \lambda \le \re z + R < 0$ for all $\lambda$ in the path of integration.
	\end{proof}
\end{proposition}

Part (b) of the above proposition shows that the first integral in~(\ref{form:fc-representation-better}) exhibits the decay rate claimed in Remark~\ref{rem:uniform-convergence-and-decay-rate}. It therefore remains to consider the second integral. Here, the exponential term is bounded below and above, so we have to show that the difference $\calR(\lambda,A+zP) - \calR(\lambda,QAQ)Q$ converges to $0$ as $\re z \to -\infty$. To do this, we represent the difference of both resolvents as a bounded multiple of $\calR(\lambda,A+zP)P$; this latter term can then easily be seen to converge to $0$.

\begin{proposition} \label{prop:resovent-estimate-on-inner-circle}
	Let $z \in \bbC$, $|z| > 2R$.
	\begin{itemize}
		\item[(a)] There is a bounded function $M: \partial B_R(0) \to \calL(X)$, not depending on $z$, such that
			\begin{align*}
				\calR(\lambda,A+zP) - \calR(\lambda,QAQ)Q = \calR(\lambda,A+zP) P \; M(\lambda)
			\end{align*}
			for all $\lambda$ with $|\lambda| = R$.
		\item[(b)] We have $\|\calR(\lambda,A+zP)P\| \le \frac{\|A\| + \delta}{\delta} \frac{\|P\|}{|\lambda-z|}$ for all $\lambda$ with $|\lambda| = R$.
	\end{itemize}
	\begin{proof}
		(a) Using that $Q$ commutes with $\calR(\lambda,QAQ)$, we can verify by a brief computation that
		\begin{align*}
			\calR(\lambda,A+zP) -  \calR(\lambda,QAQ)Q = \calR(\lambda,A+zP) P \; \big(  I + AQ \, \calR(\lambda,QAQ) \big)
		\end{align*}
		for all $\lambda$ which are contained in the resolvent sets of both $A+zP$ and $QAQ$. Hence, we simply have to define $M(\lambda) := I + AQ \, \calR(\lambda,QAQ)$ for all $\lambda \in \partial B_R(0)$.
		
		(b) Let $|\lambda|= R$. If we choose $\alpha = \|A\| + \delta$ in Lemma~\ref{lem:resolvent-of-projection}(b) (and approximate the circle $\partial B_{r_\alpha}(0)$ from the outside) we obtain $\|\calR(\lambda,zP)\| \le \frac{1}{\|A\| + \delta}$. Plugging this into the resolvent representation formula from Proposition~\ref{prop:spectrum-of-shifted-generator}(b) yields
		\begin{align*}
			\|\calR(\lambda,A+zP)P\| \le \frac{\|A\| + \delta}{\delta} \|\calR(\lambda,zP)P\|.
		\end{align*}
		However, using formula~(\ref{form:resolvent-of-projection-short}) in Lemma~\ref{lem:resolvent-of-projection}(a), we can easily see that the operator $\calR(\lambda,zP)P$ coincides with $\frac{P}{\lambda-z}$. This proves the asserted estimate.
	\end{proof}
\end{proposition}

As a consequence we obtain the desired estimate for the second integral in formula~(\ref{form:fc-representation-better}):

\begin{proposition} \label{prop:integral-estimate-on-inner-circle}
	There is a number $C > 0$, independent of $z$, such that 
	\begin{align*}
		\| \frac{1}{2\pi i} \oint_{|\lambda| = R} e^{t\lambda}[\calR(\lambda,A+zP) - \calR(\lambda,QAQ) & Q] \, d\lambda \| \le \\
		& \le C \, Re^{T_2R} \, \frac{\|A\| + \delta}{\delta} \frac{\|P\|}{|z| - R}.
	\end{align*}
	for all $t \in [-T_2,T_2]$ and for all $z \in \bbC$ with $|z| > 2R$.
	\begin{proof}
		The assertion follows immediately from Proposition~\ref{prop:resovent-estimate-on-inner-circle} if we define $C = \sup_{|\lambda| = R} \|M(\lambda)\|$. Note that if $M$ is chosen as in the proof of Proposition~\ref{prop:resovent-estimate-on-inner-circle}, then $C$ has the value $C = \sup_{|\lambda| = R} \|I + AQ \, \calR(\lambda,QAQ)\|$.
	\end{proof}
\end{proposition}

This completes the proof of our main theorem, and the choice of the constant $C$ in the above proof also gives us the estimate claimed in Remark~\ref{rem:uniform-convergence-and-decay-rate}.

It is interesting to note that the convergence of the first integral in~(\ref{form:fc-representation-better}) is due to the decay of the exponential function as $\re z \to -\infty$, while the convergence of the second integral only relies on the decay of the difference $\calR(\lambda,A+zP) - \calR(\lambda,QAQ) Q$ as $|z| \to \infty$.

\section{Operator norm convergence in Matolcsi's and Shvidkoy's Theorem} \label{section:lifting-to-operator-norm-convergence}

In this final section we briefly demonstrate that in Theorem~\ref{thm:matolcsi-shvidkoy} the convergence happens in fact with respect to the operator norm. At first glance, one might expect that we have to reprove the theorem, possibly by another method or with better estimates, to obtain this result. But in fact, things are much easier: we will use a simple lifting argument to derive operator norm convergence from the fact that we already know about the strong convergence. By the way, the theorem of course holds for negative times $t$, too.

\begin{theorem} \label{thm:lifting-to-operator-norm-convergence}
	Let $A$ be a bounded linear operator and $Q$ a bounded linear projection on a Banach space $X$. For all $t \in \bbR$ we have
	\begin{align*}
		\big( e^{\frac{t}{k}A} Q \big)^k \to e^{tQAQ}Q \quad \text{as } k \to \infty
	\end{align*}
	with respect to the operator norm; moreover, the convergence is uniform with respect to $t$ in bounded subsets of $\bbR$.
	\begin{proof}
		First note that Theorem~\ref{thm:matolcsi-shvidkoy} obviously remains true for $t \in \bbR$ (consider the negative generator $-A$ to handle the case of negative times).
		
		Denote by $\overline{B}_X$ the closed unit ball in $X$ and let $\hat X := \ell^\infty(\overline{B}_X;X)$ be the space of all bounded maps from $\overline{B}_X$ to $X$. Clearly, $\hat X$ is a Banach space when endowed with the supremum norm. For every $B \in \calL(X)$, define $\hat B \in \calL(\hat X)$ by $\hat B(y_x) := (By_x)$ for all $(y_x) = (y_x)_{x \in \overline{B}_X} \in \hat X$; the mapping
		\begin{align*}
			\calL(X) \to \calL(\hat X), \quad B \mapsto \hat B
		\end{align*}
		is an isometric unital Banach algebra homomorphism. In particular, $\hat Q$ is a projection on $\hat X$.
		
		Fix $T > 0$ and let $\hat x: = (x)_{x \in \overline{B}_X} \in \hat X$ be the identity map from $\overline{B}_X$ to $X$. For every $B \in \calL(X)$ we have $\|B\| = \|\hat B\hat x\|$. Hence,
		\begin{align*}
			& \sup_{t \in [-T,T]} \| \big(e^{\frac{t}{k}A}Q\big)^k - e^{tQAQ}Q \| = \sup_{t \in [-T,T]} \|\big( e^{\frac{t}{k}\hat A} \hat Q\big)^k \hat x- e^{t\hat Q \hat A\hat Q}\hat Q \hat x\|,
		\end{align*}
		and the latter term converges to $0$ as $k \to \infty$ according to Theorem~\ref{thm:matolcsi-shvidkoy} (respectively, its version for $t \in \bbR$).
	\end{proof}
\end{theorem}

\bibliographystyle{plain}
\bibliography{}

\begin{thebibliography}{1}

\bibitem{Arendt1993}
W.~Arendt and C.~J.~K. Batty.
\newblock Absorption semigroups and {D}irichlet boundary conditions.
\newblock {\em Math. Ann.}, 295(3):427--448, 1993.

\bibitem{Arendt1997}
W.~Arendt and M.~Ulm.
\newblock Trotter's product formula for projections.
\newblock {\em Ulmer Seminare.}, pages 394--399, 1997.

\bibitem{Exner2007}
Pavel Exner.
\newblock Unstable system dynamics: do we understand it fully?
\newblock {\em Rep. Math. Phys.}, 59(3):351--363, 2007.

\bibitem{Facchi2010}
Paolo Facchi and Marilena Ligab{\`o}.
\newblock Quantum {Z}eno effect and dynamics.
\newblock {\em J. Math. Phys.}, 51(2):022103, 16, 2010.

\bibitem{Matolcsi2003a}
M{\'a}t{\'e} Matolcsi.
\newblock On the relation of closed forms and {T}rotter's product formula.
\newblock {\em J. Funct. Anal.}, 205(2):401--413, 2003.

\bibitem{Matolcsi2004}
M{\'a}t{\'e} Matolcsi.
\newblock On quasi-contractivity of {$C_0$}-semigroups on {B}anach spaces.
\newblock {\em Arch. Math. (Basel)}, 83(4):360--363, 2004.

\bibitem{Matolcsi2003}
M{\'a}t{\'e} Matolcsi and Roman Shvidkoy.
\newblock Trotter's product formula for projections.
\newblock {\em Arch. Math. (Basel)}, 81(3):309--317, 2003.

\bibitem{Voigt1986}
J{\"u}rgen Voigt.
\newblock Absorption semigroups, their generators, and {S}chr\"odinger
  semigroups.
\newblock {\em J. Funct. Anal.}, 67(2):167--205, 1986.

\bibitem{Voigt1988}
J{\"u}rgen Voigt.
\newblock Absorption semigroups.
\newblock {\em J. Operator Theory}, 20(1):117--131, 1988.

\end{thebibliography}

\end{document}